\pdfoutput=1
\documentclass[10pt]{extarticle}

\usepackage[left=1in,right=1in,top=1in,bottom=1.3in]{geometry}
 
\usepackage{amsmath, amssymb, amsthm, latexsym, amsfonts, indentfirst, xcolor, mathtools, manfnt, microtype, subcaption}
\usepackage[utf8]{inputenc} 
\usepackage[english]{babel}

\usepackage[breaklinks]{hyperref}
\usepackage{cleveref}
\hypersetup{
	colorlinks = true, 
	urlcolor = cyan, 
	linkcolor = teal, 
	citecolor = cyan 
}

\usepackage{tikz} 

\let\OLDthebibliography\thebibliography
\renewcommand\thebibliography[1]{
	\OLDthebibliography{#1}
	\setlength{\parskip}{0pt}
	\setlength{\itemsep}{0pt plus 0.3ex}
}

\newtheorem{Theorem}{Theorem}

\newtheorem{Lemma}{Lemma}

\newtheorem{Question}{Question}
\newtheorem{Proposition}{Proposition}

\newcommand{\sm}{\!\setminus\!}

\newcommand{\N}{\mathbb N}

\newcommand{\cP}{\mathcal P}

\newcommand{\C}{\mathcal{C}}

\newcommand{\cF}{\mathcal{F}}
\newcommand{\dist}{{\rm dist}}
\newcommand{\fm}{{\rm f}_{\rm max}}
\newcommand{\ex}{{\rm ex}}
\newcommand{\exP}{{\rm ex}_\mathcal{P}}
\newcommand{\length}{\ell}
\newcommand{\CC}{{\C_{<\length}}}
\newcommand{\m}{maximal $\CC$-free plane}
\newcommand{\ms}{maximal $\CC$-free graph embedded on $\Sigma$}
\newcommand{\lens}{lens}
\newcommand{\convex}{convex}
\newcommand{\plane}{\mathcal{P}}

\newcommand{\floor}[1]{\left\lfloor{#1}\right\rfloor}
\newcommand{\ceil}[1]{\left\lceil{#1}\right\rceil}

\begin{document}

\date{}
\title{Faces in girth-saturated graphs on surfaces}
\author{
	Maria Axenovich\thanks{Karlsruhe Institute of Technology, Karlsruhe, Germany, \href{mailto:maria.aksenovich@kit.edu}{\tt maria.aksenovich@kit.edu}.}
	\and Leon Kie\ss le\thanks{Karlsruhe Institute of Technology, Karlsruhe, Germany, \href{mailto:leon.kiessle@student.kit.edu}{\tt leon.kiessle@student.kit.edu}.}
	\and Arsenii Sagdeev\thanks{Karlsruhe Institute of Technology, Karlsruhe, Germany, \href{mailto:sagdeevarsenii@gmail.com}{\tt sagdeevarsenii@gmail.com}.}
	\and Maksim Zhukovskii\thanks{School of Computer Science, University of Sheffield, UK, \href{mailto:m.zhukovskii@sheffield.ac.uk}{\tt m.zhukovskii@sheffield.ac.uk}.}
}

\maketitle

\begin{abstract}
	What is the maximum length $\fm(\ell, \Sigma)$ of a facial cycle of an inclusion-maximal graph with girth at least $\ell$ embedded on a given surface $\Sigma$? If $\Sigma=\plane$ is a plane, we show that $3\ell-11\leq \fm(\ell, \plane)\leq 8\ell-13$. We also prove that $\fm(\ell, \Sigma)$ is bounded for any integer $\ell$ and any closed surface $\Sigma$. 
	For a fixed $\Sigma$, we show that  $\Omega(\ell) =\fm(\ell, \Sigma) = O(\ell^2)$, while for a fixed $\ell\ge 6$, $\fm(\ell, \Sigma)=\Theta(g)$, where $g$ is the genus of $\Sigma$.
\end{abstract}

\section{Introduction}

Tur\'an-type problems play a substantial role in combinatorics since their introduction by Mantel~\cite{Man07} and Tur\'an~\cite{Tur41} in the first half of the 20th century. Perhaps the most extensively studied question of this type is the following. For a given family $\cF$ of graphs, what is the largest possible number of edges $\ex(n,\cF)$ in an $n$-vertex \textit{$\cF$-free} graph, that is, a graph that does not contain any $F \in \cF$ as a subgraph? Let  $\CC$ be a family of cycles of length less than $\length$.  It is known that, for a fixed $\length$, $\ex(n,\CC)=O(n^{1+1/\floor{(\length-1)/2}})$, and that the bound is asymptotically tight for some small values of $\length$, see \cite{AHL02,DB91, FS}.

Problems of this kind have a rich history of study and numerous variations, see surveys~\cite{FS, Sid95,Ver16}. Another variation was suggested by Dowden~\cite{Dow16}, who asked for the largest possible number of edges $\exP(n,\cF)$ in an $n$-vertex \textit{plane} $\cF$-free graph, where $\cP$ stands for the plane. A direct application of Euler's formula yields that $\exP(n,\CC) \le \frac{\length}{\length-2}(n-2)$ for all $n\ge \length \ge 3$ which is essentially tight. For more partial results on Dowden's problem, we refer the reader to~\cite{CLLS22,GGMPX22,GLZ24,LSS19,SWY23,SWY24} and the references therein.

Here we consider another natural parameter of graphs on surfaces: the length $\fm(G)$ of a longest facial cycle of a graph $G$ embedded on a connected surface $\Sigma$. If $G$ has no facial cycles, i.e. if the boundary of every face is either not connected or contains a vertex of multiplicity at least $2$, then we define $\fm(G)=0$. We refer the reader to a classical book~\cite{MT01} by Mohar and Thomassen for basic definitions. We say that $G$ is \textit{maximal $\CC$-free graph embedded on $\Sigma$} if $G$ is $\CC$-free, but adding any new edge to $G$ with both endpoints in $G$ creates either a crossing on $\Sigma$ or a cycle of length less than $\length$. Further, let
\begin{equation*}
	\fm(\length, \Sigma) = \max\{\fm(G):  ~ G \mbox{ is a maximal $\CC$-free graph embedded on $\Sigma$}\}.
\end{equation*}

By considering a cycle of length $2\length-3$ bounding a disc on a surface $\Sigma$, we immediately see that $\fm(\length, \Sigma)\ge 2\length-3$ for each surface $\Sigma$ and $\length \ge 3$.  Here, we show that $\fm(\length, \plane)=2\length-3$ for $\length=3,4,5$, while the first author, Ueckerdt, and Weiner \cite[Lemma~5]{AUW} proved this for $\length=6$. Our first main result provides general upper and lower bounds on $\fm(\length, \plane)$. 
In particular, it implies that $\fm(\length, \plane)$ is finite and strictly larger than $2\length-3$ for all $\length\ge 7$.

\begin{Theorem}\label{thm1}
	If $3 \le \length \le 6$, then  $\fm(\length, \plane)=2\length-3$. For any $\length\geq 7$, we have
	\begin{equation*}
		3\length-11  \leq \fm(\length, \plane) \leq 8\length-13.
	\end{equation*}
 	Moreover, if $7 \le \length \le 9$, then $\fm(\length, \plane) \ge 3\length-9$.
\end{Theorem}

Each connected \textit{closed} surface, namely a compact surface without a boundary, is homeomorphic to either a sphere with $g$ handles $\mathbb{S}_g$ or to a sphere with $g$ crosscaps $\N_g$ for some non-negative integer $g$, see e.g. \cite[Theorem~3.1.3]{MT01}.
Our second main result provides a general upper bound on $\fm(\length, \Sigma)$ for these surfaces as well as a lower bound greater than the trivial lower bound $2\ell -3$ for most of the values of $g$ and $\length$.

\begin{Theorem}\label{thm_orient_surf}
	Let $g\ge 1$ and $\ell\geq 3$  be  integers  and $\Sigma$ be a surface,  $\Sigma \in \{\mathbb{S}_g, \N_g\}$. Then
	\begin{equation*}
		g(\ell-5) \le \fm(\length, \Sigma) \le 24(2g+1)\ell^2.
	\end{equation*}
\end{Theorem}

We remark that $\fm(\length, \Sigma') = \fm(\length, \Sigma)$ for any $\Sigma \in \{\mathbb{S}_g, \N_g\}$ and any $\Sigma'$ that is homeomorphic to a surface obtained from $\Sigma$ by removing a finite number of points or discs, because any finite graph $G$ can be embedded on $\Sigma'$ if and only if it can be embedded on $\Sigma$. This implies that the bounds in \Cref{thm_orient_surf} hold for many other surfaces, and that, in particular, for a sphere $\mathbb{S}_0$ we have $\fm(\length, \plane)=\fm(\length, \mathbb{S}_0)$.


\vspace{2mm}

\noindent \textbf{Paper outline.}
In \Cref{def}, we introduce our notations and make some basic observations. We prove Theorem~\ref{thm1} for $\length\geq 7$ in \Cref{S2}, while in \Cref{SA1}, we deal with the remaining small values of $\ell$. In \Cref{Ssurf}, we prove \Cref{thm_orient_surf}.
We present an alternative, weaker but shorter argument for an arbitrary closed surface in \Cref{SA}.
Finally, we give concluding remarks and state open problems in \Cref{S3}.

\vspace{2mm}

\noindent \textbf{Note added.}
Shortly after we published a previous version of this paper, P\'alv\"olgyi and Z\'olomy~\cite{PZ25} independently showed that $\fm(\length, \plane) = O(\ell)$. 

\section{Definitions and basic observations}\label{def}

We denote the set of integers $\{1, \ldots, n\}$ by $[n]$. All graphs in this paper are finite. For a graph $G=(V, E)$,  we denote the number of its edges by $\|G\|$. We denote a cycle of length $n$ by $C_n$. When clear from the context, we shall identify a graph embedded on a surface with its embedding. In particular, we identify a \textit{planar} graph with the corresponding \textit{plane} one. Let $G$ be a graph embedded on a connected closed surface $\Sigma$. We call the connected components of $\Sigma\sm G$ \textit{faces of $G$}, see~\cite[Section~3.1.4]{gross2001topological}. If a cycle in $G$ forms a boundary of a face, we call the cycle \textit{facial}. For all standard graph theoretic notions, we refer the reader to a book by Diestel~\cite{D}.

\vspace{2mm}

For a walk $Q$, we denote its length, i.e., the number of its edges counting repetitions, by $\|Q\|$. We say that a walk is \textit{non-trivial} if the underlying graph contains a cycle. Equivalently, a walk is non-trivial if it contains at least one edge of odd multiplicity.

\vspace{2mm}

Let us denote the length of a shortest $x,y$-path in $G$ by $\dist_G(x,y)$ or simply  $\dist(x,y)$ when $G$ is clear from the context. For a path $P$ and two of its vertices $x$ and $y$, we write $xPy$ to denote the subpath of $P$ with the endpoints $x$ and $y$. We also concatenate these objects in a straightforward manner, e.g. $xPyQz$ stands for the walk that consists of the path  $xPy$ followed by the path $yQz$.

\vspace{2mm}

We call proper connected subgraphs of a cycle its \textit{segments}. Two vertices of a cycle $C$ are called {\it antipodal} if the distance between them on the cycle is $\lfloor \|C\|/2 \rfloor$. We call a path $P$ with  at least two vertices an {\it ear} of a cycle $C$ if the endpoints of $P$ are vertices of $C$ and no other vertex or edge of $P$ belongs to $C$. Observe that every path $P$ with the endpoints on a cycle $C$ is an edge-disjoint union of segments of the cycle $C$ and its ears. We say that $P$ is \textit{$C$-\convex} if there is at most one such ear. More formally, for a cycle $C$ and two of its vertices $x$ and $y$, we say that an $x,y$-path $P$ is \textit{$C$-\convex} if either $P$ is a segment of $C$ or for some vertices $x'$ and $y'$ on $P$, both $xPx'$ and $y'Py$ are segments of $C$ and $x'Py'$ is an ear of $C$. Note that $x'$ and $y'$ may coincide with $x$ and $y$, respectively, see \Cref{fig_convex0}.

\vspace{2mm}

We say that vertices $x$ and $y$ {\it split} a cycle $C'$ into two edge-disjoint $x,y$-paths $P'$ and $P''$ if $C'=P'\cup P''$. For a cycle $C'$ of length at most $2\length-2$ and two of its vertices $x$ and $y$, we call $C'$ an \textit{$(x,y;\ell)$-\lens}, or simply \textit{$x,y$-\lens} when the value of $\ell$ is clear from the context. If $x$ and $y$ split an $x,y$-\lens{} $C'$ into two $x,y$-paths that are both $C$-\convex{} for a cycle $C$, we say that $C'$ is \textit{$C$-\convex} as well, \Cref{fig_convex1}.

\begin{figure}[htb]
	\centering
	\begin{subfigure}[b]{.49\linewidth}
		\centering
		\includegraphics[scale=0.86]{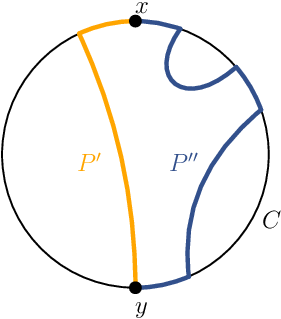}
		\captionsetup{justification=centering}
		\caption{An $x,y$-path $P'$ is $C$-convex, while $P''$ is not. \\ \ }
		\label{fig_convex0}
	\end{subfigure}
	\begin{subfigure}[b]{.49\linewidth}
		\centering
		\includegraphics[scale=0.86]{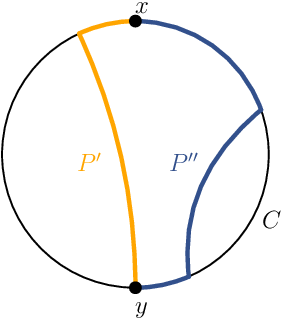}
		\captionsetup{justification=centering}
		\caption{Both $P'$ and $P''$ are $C$-convex. The cycle $C'=P'\!\cup\! P''$ \\ is a $C$-convex $(x,y;\ell)$-\lens{} if $\|C'\|\le 2\length-2$.}
		\label{fig_convex1}
	\end{subfigure}
	\captionsetup{justification=centering}
	\caption{An illustration to the definitions of $C$-convex path and lens.}
	\label{fig_convex}
\end{figure}

\vspace{2mm}

We say that a graph is a {\it subdivided wheel} if it is a union of a cycle $C$, called the {\it outer cycle of the wheel}, and a tree $T$ with exactly one vertex $c$, called the \textit{center of the wheel}, of degree greater than $2$ such that a vertex of $T$ belong to $C$ if and only if it is a leaf of $T$. We say that a path in $T$ connecting $c$ to a leaf of $T$ is a {\it spoke of the wheel} and a path in $C$ connecting two consecutive leaves of $T$ is a {\it segments of the wheel}.


\vspace{2mm}

For a graph $G$, a vertex set $W \subseteq V(G)$ and $w \in W$, we say that a tree $T \subseteq G$ is a \textit{$(W,w)$-tree} if $W \subseteq V(T)$, the leaf-set of $T$ is contained in $W$, and for each $w' \in W$, the $w,w'$-path in $T$ is a shortest $w,w'$-path in $G$.

\vspace{2mm}


We shall repeatedly use the following observations.
\begin{Lemma} \label{cl0}
	Let $\ell \ge 3$ be an integer, $\Sigma$ be a connected closed surface, $G$ be a maximal $\CC$-free graph embedded on $\Sigma$, and $C$ be its  facial cycle. If $x$ and $y$ are two vertices of $C$, then $\dist_G(x,y) \leq \length-2$. 
\end{Lemma}
\begin{proof}
	Let $x$ and $y$ be two vertices of $C$. If $x$ and $y$ are not adjacent,  adding the edge $xy$ to $G$ inside a face bounded by $C$ does not create a crossing. Now the maximality of $G$ implies that this new edge $xy$  belongs to a cycle of length at most $\length-1$, and thus $\dist_G(x,y) \leq \length-2$, as desired.
\end{proof}

\begin{Lemma} \label{path_mix}
	If $z$ is a common vertex of a shortest $x,y$-path $P$ and a shortest $x',y$-path $Q$ in a graph $G$, then $Q'=x'QzPy$ gives a shortest $x',y$-path as well.
\end{Lemma}
\begin{proof}
	If $\|Q\|< \|Q'\|$, then $\|zQy\|<\|zPy\|$. Hence, an $x,y$-walk $P'=xPzQy$ is shorter than the shortest $x,y$-path $P$, a contradiction. Therefore, the $x',y$-walk $Q'$ is no longer than the shortest $x',y$-path $Q$, and thus $Q'$ gives a shortest $x',y$-path as well.
\end{proof}

\begin{Lemma} \label{Wwtree}
	For every graph $G$, $W \subseteq V(G)$ and a vertex $w \in W$, there exists a $(W,w)$-tree $T$ in $G$.
\end{Lemma}
\begin{proof}
	Construct the desired tree iteratively. Initially, the tree $T$ contains only one vertex, $w$. At each step, we take a new vertex $w' \in W\sm \{w\}$ and consider a shortest $w,w'$-path $Q$ in $G$. Let $v$ be the last point of $Q$, counting from $w$, that is also a vertex of $T$. Note that the union of $vQw'$ with the unique $w,v$-path $P$ in $T$ is also a shortest $w,w'$-path in $G$ by \Cref{path_mix} applied with  $w',w$ playing the roles of $x',y$, respectively, and $v$ playing the roles of $x,z$. We add the path $vQw'$ to the tree $T$. Once all the vertices in $W$ are processed, the tree $T$ satisfies all the desired properties by construction.
\end{proof}

\section{Proof of Theorem~\ref{thm1} for $\length\geq 7$} \label{S2}

Throughout this section, let $\ell \ge 7$ be an integer\footnote{All the arguments in this section are valid for all $\ell \ge 3$, but they give only weak bounds when $3\le \ell \le 6$. For the exact result for these values of $\ell$, see \Cref{SA1}.}, $G$ be a \m{} graph,  and $C$ be its facial cycle. We shall need several preliminary lemmas about shortest paths and \lens es in $G$.


\begin{Lemma} \label{max_geod} 
	Let $xx', yy'\in E(C)$, $Q$ be a shortest $x',y$-path and $Q'$ be a shortest $x,y'$-path in $G$. If $xx', yy' \not\in E(Q) \cup E(Q')$, then $Q$ and $Q'$ are vertex disjoint.
\end{Lemma}

\begin{proof}
	Assume  for the contrary that $Q$ and $Q'$ share a vertex $z$. Note that a closed walk $C'=xx'QzQ'x$ is non-trivial, since both $Q$ and $Q'$ do not use the edge $xx'$. By a similar argument, a closed walk $C''=yy'Q'zQy$ is non-trivial as well. Moreover, their total length is $\|C'\|+\|C''\| = \|x'Qz\|+\|zQ'x\|+ \|y'Q'z\|+\|zQy\|+2 = \|Q\|+\|Q'\|+2 \le 2\length-2,$
	where the latter inequality is due to  \Cref{cl0}. Hence, either $C'$ or $C''$ contains a cycle of length less than $\length$, a contradiction.
\end{proof}

\begin{Lemma} \label{conv_geod}
	If $x, y \in V(C)$ and $P$ is a shortest $x,y$-path in $G$, then $P$ is $C$-\convex{}.
\end{Lemma}
\begin{proof}
	Assume  that $P$ contains at least two ears of $C$. Since every subpath of $P$ is a shortest path between its endpoints, we can assume without loss of generality that our counterexample is minimal with respect to inclusion, i.e., that it contains exactly two ears of $C$ as subpaths each of which shares an endpoint with $P$.  
	Namely, for some vertices $x'$ and $y'$ on $P$, both $xPx'$ and $y'Py$ are ears of $C$ and $x'Py'$ is a segment of $C$, see \Cref{fig_lemma5}. Note that $x'$ and $y'$ may coincide.

	
	Let $R$ be the $x,y$-path in $C$ that contains $x'$ and $y'$. Consider the family $\cF$  of shortest $w,w'$-paths of the form $wRxPyRw'$, over all possible $w, w'\in V(R)$. This family $\cF$ is non-empty since it  contains $P$. Let $P'= w_xRxPyRw_y$ be a maximal element of $\cF$  ordered by inclusion. Further, let $w_x'$ be a vertex adjacent to $w_x$ on $C$ such that $w_xw_x'$ is not an edge of $P'$. Observe that $w_x' \not\in V(P')$ since otherwise the $w_x,w_y$-path $w_xw_x'P'w_y$ is shorter than $P'$, a contradiction. Similarly, let $w_y'$ be a vertex adjacent to $w_y$ on $C$ such that $w_yw_y'$ is not an edge of $P'$, and observe that $w_y' \not\in V(P')$.

Consider a shortest $w_x', w_y$-path $Q$.  We see that $Q$ intersects $w_xP'x'$ by planarity. Denote the first point of $w_xP'x'$, counting from $w_x$, that lies on $Q$ by $z$. \Cref{path_mix} implies that $\widehat{Q}=w_x'QzP'w_y$ is also a shortest $w_x',w_y$-path. In addition, the maximality of $P'$ implies that $z\neq w_x$ and thus $w_x'w_x\not\in E(\widehat{Q})$.
	Similarly, a shortest $w_y',w_x$-path $Q'$ intersects $w_yP'y'$, and we denote the first point of $w_yP'y'$, counting from $w_y$, that lies on $Q'$ by $z'$. Then $\widehat{Q}'=w_y'Q'z'P'w_x$ is also a shortest $w_y', w_x$-path and $w_y'w_y\not\in E(\widehat{Q}')$.
	
	On the one hand, note that both $\widehat{Q}$ and $\widehat{Q}'$ contain $x'Py'$ as a subpath. On the other hand, \Cref{max_geod} 
 implies that $\widehat{Q}$ and $\widehat{Q}'$ are vertex disjoint, a contradiction.
\end{proof}

\vspace{-1.3mm}

\begin{figure}[htb]
	\centering
	\includegraphics[scale=0.9]{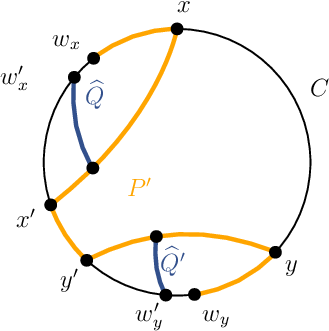}
	\caption{An illustration to the proof of \Cref{conv_geod}.}
	\label{fig_lemma5}
\end{figure}

\begin{Lemma} \label{conv_lens} 
	For every $x,y \in V(C)$ such that $\dist_C(x,y)> \ell-2$, there exists a $C$-\convex{} $(x,y; \ell)$-\lens.
\end{Lemma}

\begin{proof}
	
	Let $P$ be a shortest $x,y$-path. By  \Cref{cl0},  $\|P\| \leq \ell-2$, and thus  $P$  is not a segment of $C$. \Cref{conv_geod} implies that $P$ is $C$-\convex{}, and thus for some vertices $x'$ and $y'$ on $P$, both $xPx'$ and $y'Py$ are segments of $C$ and $x'Py'$ is an ear of $C$. Note that $x'$ and $y'$ may coincide with $x$ and $y$, respectively. There are two possible cases depending on whether the set $\{x',y'\}$  separates $x$ and $y$ in $C$ or not, see \Cref{fig_lemma_51}. However, for our argument, they are treated the same way.
	 
	Consider a family of shortest $w,w'$-paths in $G$, over all possible $w, w'\in V(C)$. Let $\cF$ be its subfamily consisting of those paths that contain $P$ as a subpath. In particular, for any $P'\in \cF$ with endpoints $w, w'$, $\|P'\| = \dist_G(w, w')$.   This family $\cF$ is non-empty since it  contains $P$. Let $P'$ be a maximal element from $\cF$ ordered by inclusion. By \Cref{conv_geod}, $P'$ is $C$-\convex{}. Since $x'P'y'=x'Py'$ is an ear of $C$, all the other vertices and edges of $P'$ are in $C$.  Let $w_x$ and $w_y$ be the endpoints of $P'$, where $w_x$ is closer to $x$ than to $y$ on $P'$, see \Cref{fig_lemma_51}.
	Note that $w_x$ and $w_y$ may coincide with $x$ and $y$, respectively.
	
	\begin{figure}[htb]
		\centering
		\begin{subfigure}[b]{.49\linewidth}
			\centering
			\includegraphics[scale=0.9]{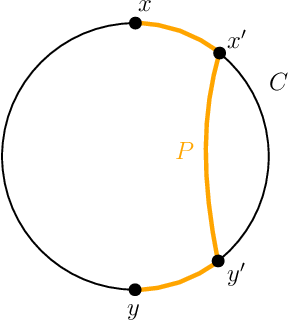}
			\captionsetup{justification=centering}
		\end{subfigure}
		\begin{subfigure}[b]{.49\linewidth}
			\centering
			\includegraphics[scale=0.9]{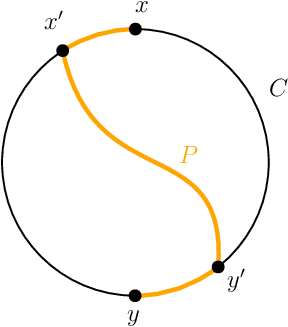}
			\captionsetup{justification=centering}
		\end{subfigure}
		\captionsetup{justification=centering}
		\caption{The set $\{x',y'\}$ can separate $x$ and $y$ in $C$ (left) or not (right).}
		\label{fig_lemma_51}
	\end{figure}
	
	Let $w_x'$ be a vertex adjacent to $w_x$ on $C$ such that $w_xw_x'$ is not an edge of $P'$. Observe that $w_x' \not\in V(P')$ since otherwise the $w_x,w_y$-path $w_xw_x'P'w_y$ is shorter than $P'$, a contradiction. Similarly, let $w_y'$ be a vertex adjacent to $w_y$ on $C$ such that $w_yw_y'$ is not an edge of $P'$, and observe that $w_y' \not\in V(P')$. Further, let $Q$ be a shortest $w_x',w_y$-path and $Q'$ be a shortest $w_y',w_x$-path. By \Cref{conv_geod}, both $Q$ and $Q'$ are $C$-\convex{}. 



	
	\vspace{2.5mm}
	
	\noindent
	{\it Case 1.  $V(Q)\cap V(w_xP'y) \subseteq \{y\}$ or  $V(Q')\cap V(w_yP'x)\subseteq \{x\}$. }
	
	\vspace{2mm}
	
	Note that this case could happen only if $\{x', y'\}$ does not separate $x$ and $y$ in $C$, see \Cref{fig_lemma_60}. If $V(Q)\cap V(w_xP'y) \subseteq \{y\}$, denote the first point of $yP'w_y$, counting from $y$, that lies on $Q$ by $z$, see \Cref{fig_lemma_61}. Observe that $z=y$ if and only if $V(Q)\cap V(w_xP'y) = \{y\}$. One can see that $C'=w_xP'zQw_x'w_x$ is a cycle of length $\|C'\|\le \|P'\|+\|Q\|+1\le 2\length-3,$ where the last inequality is by \Cref{cl0}. Hence, $C'$ is a $C$-\convex{} $x,y$-\lens, as desired. The situation when $V(Q')\cap V(w_yP'x)\subseteq \{x\}$ is symmetric.
	
	\begin{figure}[htb]
		\centering
		\begin{subfigure}[b]{.49\linewidth}
			\centering
			\includegraphics[scale=0.9]{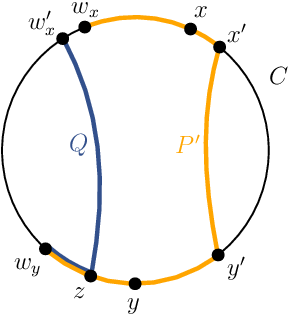}
			\captionsetup{justification=centering}
			\caption{If $V(Q)\cap V(w_xP'y) \subseteq \{y\}$, then \\ $w_xP'zQw_x'w_x$ is a $C$-\convex{} $x,y$-\lens.}
			\label{fig_lemma_61}
		\end{subfigure}
		\begin{subfigure}[b]{.49\linewidth}
			\centering
			\includegraphics[scale=0.9]{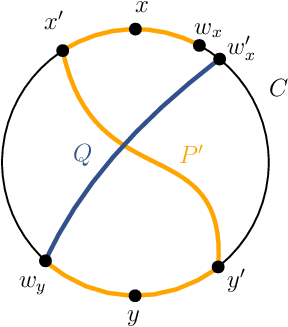}
			\captionsetup{justification=centering}
			\caption{If $\{x',y'\}$ separates $x$ and $y$ on $C$, \\ then  $V(Q)\cap V(w_xP'y) \not \subseteq \{y\}$.}
			\label{fig_lemma_60}
		\end{subfigure}
		\captionsetup{justification=centering}
		\caption{An illustration to Case 1.}
		\label{fig_lemma_52}
	\end{figure}
	
	\vspace{2.5mm}
	
	\noindent
	{\it  Case 2.  $V(Q)\cap V(w_xP'y) \not\subseteq \{y\}$ and  $V(Q')\cap V(w_yP'x) \not\subseteq \{x\}$.}
	
	\vspace{2mm}
	
	Let $z$ be the first vertex of $w_xP'y$, counting from $w_x$, that lies on $Q$. Observe that $z \neq $ Let $\widehat{Q}= w_x' Q zP'w_y$. \Cref{path_mix}  implies that $\widehat{Q}$ is a shortest $w_x',w_y$-path. Besides, $w_xw_x' \not\in E(\widehat{Q})$ by maximality of $P'$. Recall that $z \in V(w_xP'y) \sm \{y\}$ by our assumption, and thus $z \neq w_y$. Therefore, the path $\widehat{Q}$ contains the edge of $P'$ adjacent to $w_y$, and thus $w_yw_y' \not\in E(\widehat{Q})$. Similarly, let $z'$ be the first vertex of $w_yP'x$, counting from $w_y$, that lies on $Q'$. Then $\widehat{Q}'=w_y' Q' z'P'w_x$ is a shortest $w_y',w_x$-path and $w_xw_x', w_yw_y' \notin E(\widehat{Q}')$. Now \Cref{max_geod} applied to $\widehat{Q}$ and $\widehat{Q}'$ implies that these two paths are vertex-disjoint.  Note that this can happen only if  $\{x', y'\}$ separates $w_x'$ and $w_y'$ on $C$, see \Cref{fig_lemma_615}. Thus $C'=w_xw_x'\widehat{Q}w_yw_y'\widehat{Q}'w_x$ is a cycle of length $\|C'\|= \|\widehat{Q}\|+\|\widehat{Q}'\|+2\le 2\length-2,$ where the last inequality is by \Cref{cl0}, see \Cref{fig_lemma_62}. In addition, both $\widehat{Q}$ and $\widehat{Q}'$ are $C$-\convex{} by \Cref{conv_geod}, and thus $C'$ is a desired $C$-\convex{} $x,y$-\lens.	
\end{proof}

\begin{figure}[htb]
	\centering
	\begin{subfigure}[b]{.49\linewidth}
		\centering
		\includegraphics[scale=0.9]{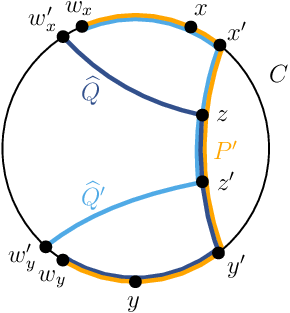}
		\captionsetup{justification=centering}
		\caption{If $\{x',y'\}$ does not separate $w_x'$ and $w_y'$ on $C$, \\ then $\widehat{Q}$ and $\widehat{Q}'$ cannot be vertex-disjoint.}
		\label{fig_lemma_615}
	\end{subfigure}
	\begin{subfigure}[b]{.49\linewidth}
		\centering
		\includegraphics[scale=0.9]{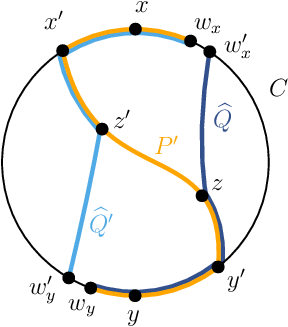}
		\captionsetup{justification=centering}
		\caption{If $\widehat{Q}$ and $\widehat{Q}'$ are vertex-disjoint, then \\ $w_xw_x'\widehat{Q}w_yw_y'\widehat{Q}'w_x$ is a $C$-\convex{} $x,y$-\lens.}
		\label{fig_lemma_62}
	\end{subfigure}
	\captionsetup{justification=centering}
	\caption{An illustration to Case 2.}
	\label{fig_lemma_53}
\end{figure}

\begin{Lemma} \label{convex_ineq}
	Let  $x,y$ be antipodal vertices of $C$ and $z$ be a center of a segment of $C$ with endpoints $x$ and $y$.  Let $C'$ be a $C$-\convex{} $(x,y; \ell)$-\lens. If $z \in V(C')$,  then $\|C\|\leq 8\ell- 13$.
\end{Lemma}
\begin{proof}
	Let $x$ and $y$ split $C'$ into two $C$-\convex{} $x,y$-paths $P'$ and $P''$, so that $z \in V(P')$.  
	Note that either $xP'z$ or $yP'z$ is a segment of $C$ since $P'$ is $C$-\convex{}. Assume without loss of generality that $xP'z$  is a segment of $C$.   Then 
	$\|C\| \leq  4\dist_C(x,z) +3  \le  4(\|P'\|-1)+3  \le 4(\|C'\|-2)+3 \le 4(2\length- 2-2)+3= 8\ell -13.$\qedhere
\end{proof}

\vspace{0mm}

\begin{proof}[Proof of Theorem~\ref{thm1} for $\length\geq 7$] \

\vspace{2mm}

 {\it Upper bound.} Note that if $\|C\|< 2\ell$, then there is nothing to prove. So we assume that $\|C\|\ge 2\ell$. Let $w_1,w_3$ and $w_2,w_4$ be two pairs of antipodal vertices of $C$ that split $C$ into four segments of almost  equal lengths. Let $C'$ be a $C$-\convex{} $(w_1,w_3; \ell)$-\lens, and $C''$ be a $C$-\convex{} $(w_2,w_4; \ell)$-\lens, which exist by \Cref{conv_lens}. 
 Let $Q_2, Q_4$ be $w_1,w_3$-paths  such that $C'=Q_2\cup Q_4$  and  $Q_1, Q_3$ be $w_2,w_4$-paths such that $C''=Q_1\cup Q_3$.
	Moreover, assume that $w_i$ is {closer in $G$} to $Q_i$ than to $Q_{i+2}$, $i\in [4]$, {index addition modulo $4$}, see Figure \ref{fig_2lens}.

If $w_i\in V(Q_i)$ for some $i\in [4]$, then $\|C\| \le 8\length-13$ by \Cref{convex_ineq}, as desired. So from now one, we assume that $w_i\not\in V(Q_i)$ for all $i\in [4]$. 
 
	By planarity, $Q_i$ intersects $Q_{i+1}$, $i\in [4]$. 
	When $|V(Q_i)\cap V(Q_{i+1})|=1$ for each $i\in [4]$, we see that $Q_1\cup Q_2 \cup Q_3 \cup Q_4$ is a union of four edge-disjoint cycles of total length $\|C'\|+\|C''\| \le 4\ell -4$, see Figure~\ref{fig_2lens}, left. Then one of the cycles has length less than $\ell$, a contradiction. 
	
	In general, recalling that $w_i\not\in V(Q_i)$ for all $i\in [4]$, let $u_i$ be the first vertex in $Q_{i}$, counting from $w_{i-1}$, that is in  $Q_{i-1}$, see \Cref{fig_2lens}, right. 
	Consider non-trivial closed walks $U_i = w_iQ_{i+1}u_{i+1}Q_iu_iQ_{i-1}w_i$, $i\in [4]$. Note that each $Q_i$ is split into three paths by $u_i$ and $u_{i+1}$, and each of these three paths is contained in exactly one of the $U_i$'s. Therefore, $\sum_{i=1}^{4}\|U_i\|=\sum_{i=1}^{4}\|Q_i\| =\|C'\|+\|C''\| \le4\length-4,$
	and thus one of the non-trivial closed walks $U_i$ contains a cycle of length less than $\length$, a contradiction again.
	
 
 \begin{figure}[htb]
 	\centering
 	\begin{subfigure}[b]{.49\linewidth}
 		\centering
 		\includegraphics[scale=0.90]{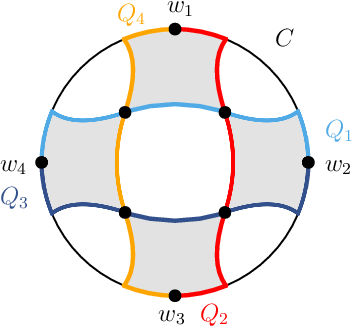}
 	\end{subfigure}
 	\begin{subfigure}[b]{.49\linewidth}
 		\centering
 		\includegraphics[scale=0.90]{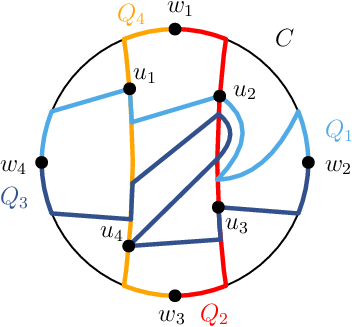}
 	\end{subfigure}
 	\caption{An illustration to the proof of \Cref{thm1}.}
 	\label{fig_2lens}
 \end{figure}

\vspace{2mm}

{\it Lower bound.} Consider the graph $W(\length)$ that is a subdivided wheel with three spokes of length $2$ each, two segment of length $\length-4$, and the third segment of length $\length-3$, see \Cref{fig1}, left. Observe that any two non-adjacent vertices of $W(\length)$ belong to some cycle of length either $2\length-4$ or $2\length-3$, and thus adding an edge between them creates a cycle of length less than $\length$. Hence, $W(\length)$ is a maximal $\CC$-free plane graph. Therefore, we have $\fm(\length)\ge \fm(W(\length))=3\length-11$, as claimed.

If $\length=7,8$, or $9$, consider a different construction $W'(\length)$, that is an edge-disjoint union of $C_9$ and $C_{3\length-9}$ that share three vertices equidistant on each of the cycles, see Figure~\ref{fig1}, right. Any two non-adjacent vertices of $W'(\length)$ belong to a  cycle of length $2\length-3$, $\length$, or $9$. Hence, $W'(\length)$ is a maximal $\CC$-free plane graph, and so $\fm(\length)\ge \fm(W'(\length))=3\length-9$ for $\length=7,8,9$, as claimed.
\end{proof}

\begin{figure}[htb]
	\centering
	\begin{subfigure}[b]{.49\linewidth}
		\centering
		\includegraphics[scale=0.60]{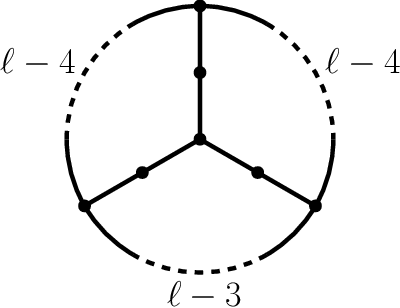}
	\end{subfigure}
	\begin{subfigure}[b]{.49\linewidth}
		\centering
		\includegraphics[scale=0.60]{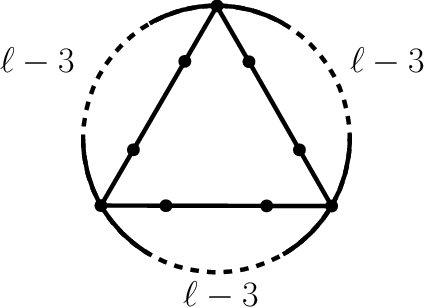}
	\end{subfigure}
	\caption{Maximal $\CC$-free plane graphs $W(\length)$ and $W'(\length)$.}
	\label{fig1}
\end{figure}

\section{Proof of Theorem~\ref{thm1} for $3\leq \length\leq 6$} \label{SA1}

In this section, we argue that $\fm(\length)=2\length-3$ if $3\le \length \le 6$.  Our proof relies on the following observation.

\begin{Lemma} \label{cl1}
	Let  $\length\geq 4$, $G$ be a \m{} graph, and $C$ be its facial cycle. If $x$ and $y$ are two non-consecutive  vertices of $C$, then $xy$ is not an edge of $G$.
\end{Lemma}
\begin{proof}
	Assume that  $x$ and $y$ are adjacent. They split $C$ into two $x,y$-paths $P'$ and $P''$, each of length at least $\length-1$, since otherwise $G$ contains a cycle of length less than $\length$. Pick two vertices, $x'$ on $P'$ and $y'$ on $P''$, such that $\dist_C(x,x')=\floor{\length/2}=\dist_C(y,y')$. Note that $\dist(x,x')=\floor{\length/2}$, since otherwise the union of $xP'x'$ and a shortest $x,x'$-path in $G$ contains a cycle of length less than $\length$. Assume that there is an $x',y$-path $Q$ of length at most $\ceil{\length/2}-2$. Then $Q$ is shorter than $xP'x'$, and thus $xP'x'Qyx$ contains a cycle of length less than $\length$. This is a contradiction implying that $\dist(x',y)\ge \ceil{\length/2}-1$. Similarly, $\dist(y,y')=\floor{\length/2}$ and $\dist(x,y')\ge \ceil{\length/2}-1$. Since each $x',y'$-path contains either $x$ or $y$ by planarity, we conclude that $\dist(x',y')\ge \floor{\length/2} + \ceil{\length/2}-1 = \length-1$, which contradicts \Cref{cl0} and thus completes the proof.
\end{proof}

\vspace{0mm}

\begin{proof}[Proof of Theorem~\ref{thm1} for $3\leq \length\leq 6$] \

\vspace{2mm}

Recall that the lower bound is immediate by considering $C_{2\ell-3}$, so we proceed with the upper bound.

\vspace{2mm}

If $\length=3$, then  the family $\CC$ of forbidden cycles is empty and any maximal plane graph is a triangulation, so $\fm(3)=3$.

\vspace{2mm}

Let $\length=4$.
Let $C=u_0u_1\cdots u_ku_0$ be a facial cycle of length at least $6$ in a \m{} graph $G$. 
By Lemmas~\ref{cl0} and~\ref{cl1}, any two non-consecutive vertices of $C$ are at distance exactly 2 in $G$. By planarity,  $u_0,u_3$- and $u_1,u_4$-paths of length 2 must share their center vertex, say $w$. Then $u_0wu_1$ is a cycle of length $3$, a contradiction. Hence, $\fm(4)\le 5= 2\ell-3$.

\vspace{2mm}

Let  $\length=5$.
Let $C=u_0u_1\cdots u_ku_0$ be a facial cycle of length at least $8$ in a \m{} graph $G$. 
In this case,  $k\geq 7$ and every two non-consecutive vertices of $C$ are at distance either 2 or 3 in $G$ by Lemmas~\ref{cl0} and~\ref{cl1}. 

Assume first that $\dist(x,y)=3$ for all vertices $x, y$ of $C$ such that $\dist_C(x,y)\ge 3$. This implies that the shortest $u_0,u_4$- and $u_1,u_5$-paths $P$ and $P'$  have no inner vertices on $C$. By planarity, they share a vertex. Thus $u_0Pu_4u_5P'u_1u_0$ is an edge-disjoint union of two nontrivial closed walks of total length $8$. Thus there is a cycle of length at most $4$, a contradiction.

Now assume that $\dist(x,y)=2$ for some vertices $x, y$ of $C$ such that $\dist_C(x,y)\ge 3$. Assume without loss of generality that $\dist(u_2,u_j)=2$ for some $5 \le j \le \ceil{k/2}+2$, and thus $u_2wu_j$ is an ear of $C$ for some vertex $w$ by \Cref{cl1}, see \Cref{fig_l51}. Note that $\dist(u_0,u_2)=2$ and $\dist(u_0,w)\ge 2$, since otherwise $u_0w$ is an edge and $u_0u_1u_2wu_0$ is a cycle of length $4$. Similarly, $\dist(u_4,u_2)=2$ and $\dist(u_4,w)\ge 2$. In addition, $u_j$ and $u_0$ are not consecutive on $C$ since $j+1 \le \ceil{k/2}+3 \leq k$, and thus $\dist(u_0,u_j)\ge 2$. Since the shortest $u_0,u_4$-path of length at most $3$ contains one of the vertices $u_2,w,u_j$ by planarity, we conclude that this vertex is $u_j$ and $\dist(u_0,u_j)=2$, $\dist(u_4,u_j)=1$. The latter equality implies that $j=5$, and thus $\dist_C(u_0,u_j)>2$. Now the equality $\dist(u_0,u_j)=2$ and \Cref{cl1} imply that there exists a vertex $w'$ such that $u_0w'u_j$ is an ear of $C$. Moreover, $w'\neq w$ since $\dist(u_0,w)\geq 2$.

Repeating the argument from the previous paragraph verbatim replacing the vertices $u_7,u_6,u_5,u_3,u_2$ by $u_0, u_1, u_2, u_4, u_5=u_j$, respectively, we conclude that there exists a vertex $w''$ such that $u_7w''u_2$ is an ear of $C$. Moreover, one can see that two ears $u_0w'u_5$ and $u_7w''u_2$ must share their center vertex by planarity. In other words, $w'=w''$ and thus $u_2wu_5w'u_2$ is a cycle of length $4$, a contradiction, see \Cref{fig_l52}. Hence, $\fm(5)\le 7 = 2\ell-3$, as desired.

\vspace{2mm}

The case $\length=6$ follows from Lemma~5 in~\cite{AUW}. 
\end{proof}

\begin{figure}[htb]
	\centering
	\begin{subfigure}[b]{.49\linewidth}
		\centering
		\includegraphics[scale=0.9]{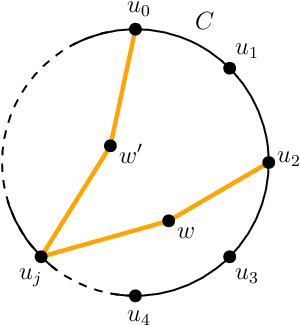}
		\caption{$u_2wu_j$ and $u_0w'u_j$ are ears of $C$.}
		\label{fig_l51}
	\end{subfigure}
	\begin{subfigure}[b]{.49\linewidth}
		\centering
		\includegraphics[scale=0.9]{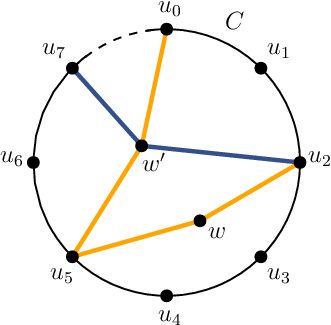}
		\caption{$u_2wu_5w'u_2$ is a cycle of length $4$.}
		\label{fig_l52}
	\end{subfigure}
	\captionsetup{justification=centering}
	\caption{An illustration to the case $\ell=5$.}
	\label{fig_l5}
\end{figure}

\section{Proof of Theorem~\ref{thm_orient_surf}} \label{Ssurf}

Throughout this section, let $g \ge 1$ and $\ell\ge 3$ be integers, $\Sigma \in \{\mathbb{S}_g, \N_g\}$, $G$ be a maximal $\CC$-free graph embedded on $\Sigma$, and $F$ be its face bounded by a cycle $C$. Whenever it does not cause confusion, we identify a graph and its embedding on $\Sigma$. In particular, we identify a path (or a cycle) and the corresponding simple (closed) curve on $\Sigma$. We shall use the following properties of the shortest paths in $G$.

\begin{Lemma} \label{center_avoiding}
	For any segment $S$ of $C$ of order $\ell$ and any $z\in V(G)$ there is $x\in V(S)$ such that $\dist(x,z)> \frac{\ell}{2}-1$.
\end{Lemma}

\begin{proof}
	Let $S=x_1\cdots x_{\ell}$. Assume for the contrary that for each $i \in [\ell]$, $d(i) \coloneqq \dist_G(x_i,z) \le \frac{\ell}{2}-1$. For each $i \in [\ell]$, let $Q_i$ be a shortest $x_i,z$-path in $G$. Note that if $i<\ell$, then $x_iQ_izQ_{i+1}x_{i+1}x_i$ is a closed walk of length at most $2(\frac{\ell}{2}-1)+1 = \ell-1$. Hence, this walk is trivial, i.e., either $Q_i= x_ix_{i+1}Q_{i+1}z$ (in which case $d(i)=d(i+1)+1$)  or $Q_{i+1}= x_{i+1}x_iQ_iz$ (in which case $d(i+1)=d(i)+1$). Assume that there is a local maximum of $d(\cdot)$ at $i \in [\ell]$, i.e., that  $d(i-1)=d(i)-1=d(i+1)$ for some $i-1,i,i+1 \in [\ell]$. Then $Q_i= x_ix_{i-1}Q_{i-1}z = x_ix_{i+1}Q_{i+1}z$. This is a contradiction, since $Q_i$ has an endpoint $x_i$ and two edges incident to it. Therefore, the function $d(\cdot)$ has a unique minimum on $[\ell]$, it first strictly decreases and then strictly increases. However, this implies that $\ell = |[\ell]| \le d(1)+d(\ell)+1 \le 2(\frac{\ell}{2}-1)+1=\ell-1$, a contradiction.
\end{proof}

\begin{Lemma} \label{remote_vertex}
	Let $p,q \in V(C)$, and $\gamma$ be a simple curve on $F$ with endpoints $p, q$. Let $\gamma'$ be a simple curve on $\Sigma$ with endpoints $p, q$ that consists of vertex disjoint paths in $G$ with the endpoints on $C$ connected by simple curves on $F$. Suppose that the interiors of $\gamma$ and $\gamma'$ are disjoint and that $\Sigma\sm (\gamma \cup \gamma')$ consists of two connected components. Let $D$ be one of these components and suppose that $D\cup\gamma \cup \gamma'$ is homeomorphic to a closed disk. Let $S$ be a $p,q$-path in $C$ and suppose that, as a curve, $S$ lies in $D\cup \gamma'$. Let $X=V(S)$ and $Z=\gamma'\cap V(G)$. If $|X|\ge \ell|Z|$, then there exists $x\in X$ such that $\dist_G(x,Z)> \frac{\ell}{2}-1$.
\end{Lemma}


\begin{proof}
	Assume the contrary, namely that $\dist_G(x,Z) \le \frac{\ell}{2}-1$ for all $x \in X$. Label the vertices in $X$ by  $x_1, \ldots , x_{k}$ in their relative order on $S$, and label the vertices in $Z$ by $z_1,\dots , z_{m}$ in their relative order on $\gamma'$ such that $x_1=z_1=p, x_k=z_m=q$. We order the vertices of $Z$ according to their indices, i.e., we  say that $z_j<z_{j'}$ if $j<j'$. Let
	\begin{align*}
		d(i)&=\dist(x_i, Z) = \min_{j\in [m]} \dist(x_i,z_j) \hspace{6mm}  \mbox{ and } \\
		z(i)&= z_j, \ \   \mbox{ where }  j = \min\{j':  d(i)= \dist(x_i, z_{j'})\},
	\end{align*}
	i.e., $z(i)$ is the smallest vertex from $Z$ such that $\dist(x_i, Z)= \dist(x_i, z(i))$. Let $Q_i$ be a shortest $x_i,z(i)$-path in $G$, see \Cref{fig_L10}. Recall that the interior of $\gamma$ is in $F$ and thus it contains no vertices of $G$. Hence, each $Q_i$ lies in $D$, since a path in $G$ that starts on $S$ can leave the connected component $D$ only through the set $Z$ on its boundary.
	
	Now we show that $z(\cdot)$ is non-decreasing. Assume for the contrary that $z(i+1)<z(i)$ for some $i$. Since the edges of $Q_i$ and $Q_{i+1}$ are in $D\sm F$,  by planarity, $Q_i$ and $Q_{i+1}$ share a vertex. Thus either $x_ix_{i+1} \in E(Q_i)$ or $x_{i+1}x_{i} \in E(Q_{i+1})$, since otherwise the walk $z(i)Q_ix_ix_{i+1}Q_{i+1}z(i+1)$ of length less than $\ell$ contains a cycle.
	
	
	
	If $x_ix_{i+1} \in E(Q_i)$, then $\|x_iQ_iz(i)\|= 1+\|x_{i+1}Q_iz(i)\| \ge 1+\|x_{i+1}Q_{i+1}z(i+1)\| = \|x_ix_{i+1}Q_{i+1}z(i+1)\|$ which contradicts the definition of $z(i)$, since $z(i+1)<z(i)$.
	
	Similarly, if $x_{i+1}x_{i} \in E(Q_{i+1})$, then $\|x_iQ_{i+1}z(i+1)\|=\|x_{i+1}Q_{i+1}z(i+1)\|-1\le \|x_{i+1}x_{i}Q_{i}z(i)\|-1=\|x_{i}Q_{i}z(i)\|$ which again contradicts the definition $z(i)$, since $z(i+1)<z(i)$.
	
	These contradictions imply that $z(\cdot)$ is indeed non-decreasing. Hence, for each $z \in Z$, the set
	\begin{equation*}
		I(z) \coloneqq \{i \in [k]: z(i)=z\}
	\end{equation*}
	is a (possibly empty) interval of consecutive integers, and thus $|I(z)|<\ell$ by \Cref{center_avoiding}. Therefore, $|X|= \sum_{z \in Z}|I(z)| < \ell|Z|$, which contradicts the condition $|X|\ge \ell|Z|$ and completes the proof.
\end{proof}

\begin{figure}[htb]
	\centering
	\begin{subfigure}[b]{.49\linewidth}
		\centering
		\includegraphics[scale=0.9]{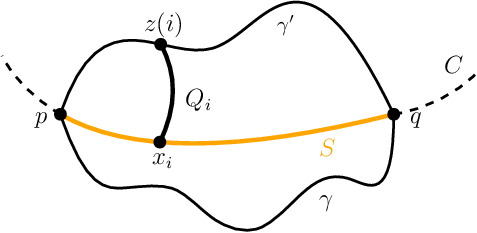}
		\captionsetup{justification=centering}
		\caption{$z(i)$ is the closest to $x_i$ point on $Z$ \\ and $Q_i$ is a shortest $x_i,z(i)$-path.}
		\label{fig_L10_1}
	\end{subfigure}
	\begin{subfigure}[b]{.49\linewidth}
		\centering
		\includegraphics[scale=0.9]{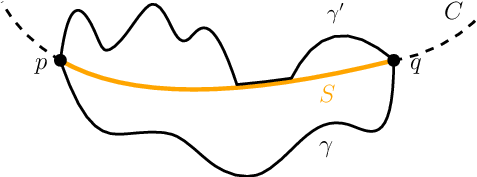}
		\captionsetup{justification=centering}
		\caption{Note that the part of $D$ bounded between \\ $S$ and $\gamma'$ is not necessarily connected.}
		\label{fig_L10_2}
	\end{subfigure}
	\captionsetup{justification=centering}
	\caption{An illustration to the proof of \Cref{remote_vertex}.}
	\label{fig_L10}
\end{figure}

\begin{Lemma} \label{2discs}
	If a graph $H$ with $v$ vertices and $e$ edges is embedded on $\Sigma \in \{\mathbb{S}_g, \N_g\}$, then at least $e-v-2g+2$ faces of $H$ are homeomorphic to open discs.
\end{Lemma}
\begin{proof}
	Let $f$ be the number of faces of $H$ and $f'$ be the number of those faces that are homeomorphic to open discs. Observe that if a face $F'$ is \textit{not} homeomorphic to an open disc, then we can connect two of the vertices of $H$ on the boundary of $F'$ by a curve $c$ inside $F'$ such that $F'\sm c$ is connected. Indeed, if the boundary of $F'$ consists of at least 2 connected components, take 2 vertices of $H$ on different components, and connect them by a simple curve in $F'$. Otherwise, if the boundary of $F'$ is connected, but $F'$ is not homeomorphic to an open disc, then the classification theorem, see e.g. \cite[Theorem~3.1.3]{MT01}, implies that there is a handle or a cross-cap in the interior of $F'$. In this case, we pick a vertex on the boundary of $F'$ and draw a loop $c$ in $F'$ through this handle or a cross-cap such that $F'\sm c$ is connected.
	
	Let $H'$ be a multigraph obtained from $H$ by adding such an edge in each of the $f-f'$ faces that are not homeomorphic to discs. Note that $H'$ has $v$ vertices, $e+f-f'$ edges, and $f$ faces by construction. By Euler's formula, see e.g.~\cite[Section~3.1]{MT01}, we have $v-(e+f-f')+f \ge 2-2g$, or, equivalently, $f'\ge e-v-2g+2$, as desired. 
\end{proof}

\vspace{0mm}

\begin{proof}[Proof of Theorem \ref{thm_orient_surf}]\
	
\vspace{2mm}

{\it Upper bound.} 	
Recall that $g \ge 1$ and $\ell\ge 3$ are integers, $\Sigma \in \{\mathbb{S}_g, \N_g\}$, $G$ is a maximal $\CC$-free graph embedded on $\Sigma$, and $F$ is its face bounded by a cycle $C$. Assume for the contrary that $\|C\|>24(2g+1)\ell^2$.

We begin by outlining the proof strategy. We define a suitable constant $s=s(g)$. 
First we construct an auxiliary graph $G'$ embedded on $\Sigma$  as follows. We consider a set $W$ of $s+1$ vertices on $C$ that are  almost equidistant  and  take a subgraph of $G$ forming a tree $T$ with leaf-set  contained in  $W$.  We consider a  new cycle $C'$ on $s+1$ vertices  that is embedded in $F$ very close to $C$ and a matching $M$ between the vertices of  $C'$ and $W$. Let $T'$ be the tree obtained from $T\cup M$ by `dissolving' vertices of degree 2.  Then define the graph $G'$ to be $C'\cup T'$. Next, we argue that there are two faces $D_1, D_2$ of $G'$ such that each of them, together with its boundary, is homeomorphic to a closed disc. For each $i\in \{1,2\}$, we prove that there is a vertex $x_i\in V(C)$ in $D_i$ such that its distance to the set $Z_i$ of vertices of $G$ on the boundary of $D_i$ satisfies $\dist_G(x_i,Z_i)> \frac{\ell}{2}-1$. Since any shortest $x_1,x_2$-path  in $G$  passes through the set $Z_1$ and then through the set $Z_2$, we conclude that $\dist_G(x_1,x_2)> \ell -2$. This contradiction to \Cref{cl0} completes the proof of the upper bound $\|C\|\le 24(2g+1)\ell^2$.

Now we start a formal proof. 

\vspace{2mm}
{\it Step 1:  constructing the auxiliary graph $G'$.}
Let $s = 6(2g+1)-1$.  Let $W=\{w_0, \ldots, w_s\}$ be a set of  $s+1$ vertices on $C$ that appear in the order of their indices on $C$ and split this cycle into $s+1$ segments, each containing at least $\|C\|/(s+1)$ vertices. Let $T$ be a $(W,w_0)$-tree that exists by \Cref{Wwtree}.


Consider a cycle $C'$ on $s+1$ vertices labeled $w_0', \ldots, w_s'$ in order and embed $C'$ in the face $F$ sufficiently close to the cycle $C$ such that $F\sm C'$ consists of two connected components and, moreover, the component that is bounded between $C$ and $C'$ is homeomorphic to an annulus. The existence of such an embedding follows from the tubular neighborhood theorem, see \cite[Theorem~5.2]{H76}. Consider a matching $M$ with edges   $w_iw_i'$, $0\leq i\leq s$, and  draw the edges of $M$  in the annulus between $C$ and $C'$ such that the respective $s+1$ curves are pairwise disjoint, see \Cref{fig_G}, middle.

Let $T'$ be the tree obtained by `dissolving' degree 2 vertices in the tree $T\cup M$. In other words, let $T'$ be a unique tree without degree 2 vertices embedded on $\Sigma$ such that $T\cup M$ is a subdivision of $T'$ and, as a curve, each edge of $T'$ is a union of some edges of $T\cup M$. Note that $T'$ has precisely $s+1$ leaves  $w_i'$, $0\leq i\leq s$.

Now, define the graph $G'$ to be $T'\cup C'$ and  fix its embedding on $\Sigma$ as described, see \Cref{fig_G}, right.

\vspace{2mm}
{\it Step 2:  defining faces $D_1$ and $D_2$  of $G'$.}
Let $\cF'$ be the set of all faces of $G'$ except for the face $F'$ bounded by $C'$.  
Observe that $|E(G')|-|V(G')|= (s+1)+|E(T')|-|V(T')| = s$. Hence, the number $f'$ of faces of $G'$ that are homeomorphic to an open disc satisfies $f'\ge s-2g+2$ by \Cref{2discs}.
As usual, let the \textit{degree} of a face be the number of edges on its boundary counted with multiplicities. For  $k\in \N$, let $f_k$ be the number of  faces from $\cF'$ of degree $k$ and let $f_{\ge 6} \coloneqq \sum_{k\ge 6} f_k$. The handshaking lemma implies that $6f_{\ge 6} \le \sum_{k \in \N} kf_k = 2|E(G')|-(s+1) = 2(s+1+|E(T')|)-(s+1) \le 5s-1$, where the last inequality holds because a tree with $s+1$ leaves and without vertices of degree $2$ has at most $2s-1$ edges. Therefore, $f_{\ge 6} \le (5s-1)/6$.
Let $\cF\subseteq \cF'$ be the set of faces homeomorphic to an open disc and that have degree less than $6$. Summarising the inequalities above, we conclude that $|\cF| \ge f'-f_{\ge 6}-1\ge \frac{s+1}{6}-2g+1=2$.

Note that if a face in $\cF$ is not bounded by a cycle, then at least one vertex appears at least twice on its boundary walk and splits this walk into two closed sub-walks. Since the length of the boundary walk is at most $5$, one of these sub-walks has length at most $2$. 
Observe that a sub-walk of length $1$ corresponds to a loop, and a sub-walk of length $2$ corresponds to a pendant edge or to parallel edges in $G'$.  However, $G'$ is a simple graph without pendant edges by construction. Thus all faces from $\cF$ are bounded by cycles. Fix any two of them and call these faces $D_1$ and $D_2$, see \Cref{fig_G}, right.

\vspace{2mm}
{\it Step 3: finding remote vertices $x_1$ and $x_2$ on $C$.}  
For $0 \le i \le s$, let $w_iCw_{i+1}$ be the segment of $C$ with the endpoints $w_i$ and $w_{i+1}$ that does not contain other vertices $w_j$, $j \neq i, i+1$. Here and for the rest of the proof, we treat indices modulo $s+1$. 

Since $T'$ is a tree, the boundary cycle of $D_1$ contains an edge $w'_jw'_{j+1}$ of $C'$ for some $0\le j \le s$. We claim that this boundary cycle contains neither $w'_{j-1}w'_j$ nor $w'_{j+1}w'_{j+2}$. Indeed, if, say, the boundary cycle of $D_1$ contains $w'_{j+1}w'_{j+2}$, then $w'_jw'_{j+1}w'_{j+2}$ is a subpath of the boundary cycles of both $D_1$ and the face $F'$ bounded by $C'$, and thus the degree of $w'_{j+1}$ in $G'$ equals 2, a contradiction. In particular, this implies that the curve $\gamma_1$ corresponding to the path $w_{j}w'_{j}w'_{j+1}w_{j+1}$, i.e., formed by the segment $w'_{j}w'_{j+1}$ of $C'$ and the edges $w_{j}w'_{j}, w_{j+1}w'_{j+1}$ of $M$, lies on the boundary of $D_1$. Let $\gamma_1'$ be the curve with the endpoints $w_j$ and $w_{j+1}$ that, together with $\gamma_1$, forms the boundary of $D_1$. Note that the open region of the annulus between $C$ and $C'$ bounded by $\gamma_1$ and $S_1\coloneqq w_{j}Cw_{j+1}$ is disjoint from $G$, and thus, as a curve, $S_1$ lies in $D_1\cup\gamma_1'$.

Let $X_1 = V(S_1)$ and $Z_1=\gamma_1'\cap V(G)$. We claim that $|X_1|\ge \ell|Z_1|$. Indeed, on the one hand, recall that $|X_1| \ge \|C\|/(s+1) \ge 4\ell^2$ by construction. On the other hand, since $D_1 \in \cF$, the set $Z_1$ is contained in the union of at most $4$ edges of $T'$. Each of these edges is contained in some $w_0,w$-path, $w \in W$, in a $(W,w_0)$-tree $T$, which in turn contains at most $\ell-1$ vertices of $G$ by \Cref{cl0}. Hence, $|Z_1|< 4\ell \le |X_1|/\ell$, as claimed. Now \Cref{remote_vertex} applied to $w_{j}, w_{j+1}$ and $\gamma_1, \gamma_1', D_1, S_1, X_1, Z_1$ playing the roles of $p, q$ and $\gamma, \gamma', D, S, X, Z$, respectively, 
implies that there exists $x_1 \in X_1$ such that $\dist_G(x_1,Z_1)> \frac{\ell}{2}-1$. 

Define $\gamma_2, \gamma_2', S_2, X_2, Z_2$ for the face $D_2$ in a similar fashion, and apply \Cref{remote_vertex} again to find a vertex $x_2 \in X_2$ such that $\dist_G(x_2,Z_2)> \frac{\ell}{2}-1$.

By \Cref{cl0}, there exists an $x_1,x_2$-path $P$ in $G$ of length at most $\ell-2$. On the other hand, $P$ intersects both $Z_1$ and $Z_2$ since $x_1$ and $x_2$ lie in faces with disjoint interiors $D_1$ and $D_2$, respectively, and since every path in $G$ that starts in $D_1$ (resp., in $D_2$) can leave $D_1$ (resp., $D_2$) only through the set $Z_1$ (resp., $Z_2$) on its boundary. In other words, there exist (not necessarily distinct) vertices $y_1 \in V(P)\cap Z_1$, $y_2 \in V(P)\cap Z_2$ such that $P=x_1Py_1Py_2Px_2$. Therefore, $\|P\|\ge \dist(x_1,y_1)+\dist(x_2,y_2) \ge \dist(x_1,Z_1)+\dist(x_2,Z_2)> 2(\frac{\ell}{2}-1)= \ell-2$. This contradiction to \Cref{cl0} completes the proof of the desired upper bound $\|C\|\le 24(2g+1)\ell^2$.

\begin{figure}[htb]
	\centering
	\hspace{-10mm}
	\raisebox{2mm}{
	\begin{subfigure}[c]{.3\linewidth}
		\centering
		\includegraphics[scale=0.8]{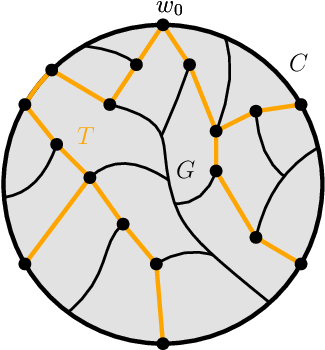}
		\label{fig_G1}
	\end{subfigure}
	}
	\hspace{-3mm}
	\begin{subfigure}[c]{.3\linewidth}
		\centering
		\includegraphics[scale=0.8]{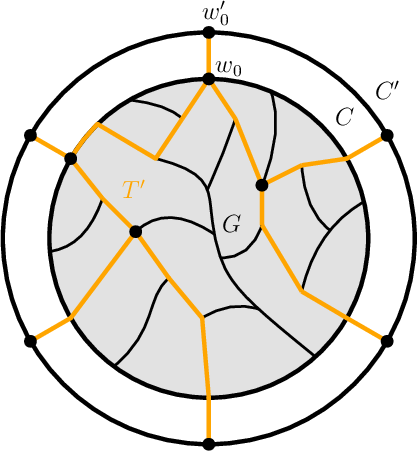}
		\label{fig_G2}
	\end{subfigure}
	\hspace{8mm}
	\begin{subfigure}[c]{.3\linewidth}
		\includegraphics[scale=0.8]{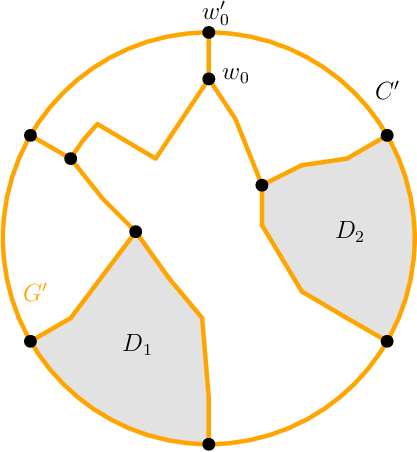}
		\label{fig_G3}
	\end{subfigure}
	\captionsetup{justification=centering}
	\caption{Auxiliary graphs $T$ (left), $T'$ (middle), and $G'$ (right).}
	\label{fig_G}
\end{figure}

\vspace{2mm}

{\it Lower bound.}
If $\ell \le 5$, then there is nothing to prove. Assume that $\ell \ge 6$ and  let $m=\ell-4$. Consider a subdivided wheel with $g+1$ spokes of length $1$ and $g+1$ segments of length $2m$ formed by a cycle $C$ and a star $T$. Let $T'$ be a star with $g+1$ leaves such that its center is not a vertex of the wheel and its leaves are central vertices of the $g+1$ segments of the wheel. Call the union of these graphs $G$.

Sketch  $G$ on the plane such that $C$ is a circle, the centers of $T$ and $T'$ lie inside $C$ and their edges are straight line segments, as shown in  \Cref{fig_LB1}. 
Note that in this sketch of $G$, precisely $g$ edges of $T'$ cross $T$. For each of these edges, we take sufficiently small circles around two of its inner points, one before the first crossing and another one after the last crossing, replace two discs bounded by these circles with a handle, and redirect the inner part of the edge along this handle. The resulting drawing of $G$ on $\mathbb{S}_g$ is crossing-free, see~\cite[Section~2]{M09}. Moreover, the cycle $C$ of length $2(g+1)m$ is still a facial cycle of our drawing corresponding to the `outer' face.



First, we observe that every cycle in $G$ must contain at least two edge-disjoint  segments of $C$ of length $m$ and at least two additional edges. In other words, $G$ contains no cycles of length less than $2m+2 \ge \ell$.

Second, we verify the maximality. Let $x,y$ be two vertices of $C$. Note that if  $\dist_C(x,y) \le m$, then adding the edge $xy$ to $G$ creates a cycle of length at most $m+1<\ell$, and there is nothing to prove. Hence, we assume without loss of generality that $\dist_C(x,y) > m$. Let $v_x$ and $w_x$ be the vertices of $T'$ and $T$, respectively, closest to $x$ on $C$. Define $v_y$ and $w_y$ in a similar way. Note that $v_x$ and $w_x $ may coincide with $v_y$ and $w_y$, respectively, but not simultaneously because $\dist_C(x,y) > m$.  Let $Q_1$ be an $x,y$-path formed by the union of the shortest $x,v_x$- and $y,v_y$-paths on $C$ with the $v_x,v_y$-path on $T'$. Similarly, let $Q_2$ be an $x,y$-path formed by the union of the shortest $x,w_x$- and $y,w_y$-paths on $C$ with the $w_x,w_y$-path on $T$, see \Cref{fig_LB2}. Observe that the union of $Q_1$ and $Q_2$ is a cycle and its length 
is at most $2m+4 \le 2\ell-3.$
Hence, adding the edge $xy$ to $G$ creates a cycle of length less than $\ell$. Therefore, $G$ is a maximal $\CC$-free graph embedded on $\mathbb{S}_g$. Now we conclude that $\fm(\ell,\mathbb{S}_g)\ge \|C\| = (2g+2)m$, which is even stronger than the desired lower bound.

For nonorientable surfaces, we cut off one or two small discs inside the face $F$ and replace them with crosscaps. Since the resulting surface is homeomorphic to $\N_{2g+1}$ or $\N_{2g+2}$, respectively, see e.g. \cite[Theorem~3.1.3]{MT01}, we conclude that $\fm(\ell,\mathbb{N}_{2g+1}),\ \fm(\ell,\mathbb{N}_{2g+2})\ge \|C\| = (2g+2)m$. It remains only to note that for $g=1,2$, the desired inequality $\fm(\ell,\mathbb{N}_{g})\ge g(\ell-5)$ follows from the lower bound $\fm(\ell,\mathbb{N}_{g})\ge 2\ell-3$ discussed in the introduction.
\end{proof}

\begin{figure}[htb]
	\centering
	\begin{subfigure}[b]{.49\linewidth}
		\centering
		\includegraphics[scale=0.9]{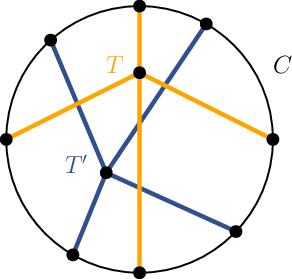}
		\caption{A drawing of $G$ on the plane in case $g=3$.}
		\label{fig_LB1}
	\end{subfigure}
	\begin{subfigure}[b]{.49\linewidth}
		\centering
		\includegraphics[scale=0.9]{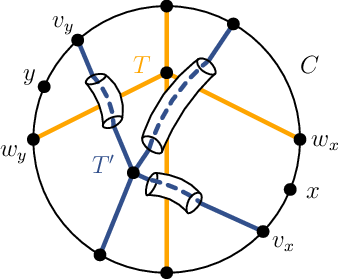}
		\caption{A drawing of $G$ on $\mathbb{S}_3$ without crossings.}
		\label{fig_LB2}
	\end{subfigure}
	\captionsetup{justification=centering}
	\caption{An illustration to the lower bound.}
	\label{fig_LB}
\end{figure}

\section{Concluding remarks} \label{S3}

We showed that the maximum length $\fm(\length, \plane)$ of a facial cycle of an inclusion-maximal plane graph with girth at least $\length$ satisfies $3\length-11 \le \fm(\length, \plane) \le 8\length-13$. We would like to pose the following question.

\begin{Question}
	Is it true that $\fm(\length, \plane) = 3\length + o(\length)$?
\end{Question}

When $\ell \ge 6$ and the plane is replaced with an arbitrary closed surface $\Sigma$ of genus $g\ge 1$, we showed that $g(\ell-4) \le \fm(\length, \Sigma) \le 24(2g+1)\ell^2$. In fact, our proof gives slightly better bounds depending on the orientability of $\Sigma$: $\fm(\length, \mathbb{N}_g) \le 24(g+1)\ell^2$ and $\fm(\length, \mathbb{S}_g) \ge (2g+2)(\ell-4)$. By replacing one of the trees with a path in our construction, we can also show that $\fm(\length, \mathbb{N}_g) \ge (2g+2)(\ell-g-2)$ for $\ell \ge 2g+3$.

Note that if $\ell \ge 6$ is fixed while $g$ tends to infinity, then $\fm(\ell, \Sigma)=\Theta(g)$ for every closed surface $\Sigma$ of genus $g$. We can also show that $\fm(3, \Sigma)=\Theta(\sqrt{g})$ as $g$ growth, by relating the value of $\fm(3, \Sigma)$ to the size of the largest clique that can be embedded on $\Sigma$. It may be interesting to determine the growth rate of $\fm(\ell, \Sigma)$ for the remaining values $\ell=4,5$ as well.

In the regime when $\Sigma$ is fixed while $\ell$ tends to infinity, we know that $\Omega(\ell) =\fm(\ell, \Sigma) = O(\ell^2)$. We believe that $\fm(\ell, \Sigma)=O(\ell)$ and wonder if the argument from \Cref{S2} or from \cite{PZ25} can be generalized to show that. Perhaps, the function $\fm(\ell, \Sigma)$ is even linear in both parameters.

\begin{Question}
	Does there exist an absolute constant $C>0$ such that $\fm(\length, \Sigma) \le Cg\ell$ for each closed surface $\Sigma$ of genus $g\ge 1$ and for each integer $\ell\ge 3$?
\end{Question}

We remark that it was not originally obvious for us that $\fm(\ell, \Sigma)$ is bounded by any function of $\ell$ even in the simplest case when $\Sigma$ is a plane. Our argument in \Cref{SA} is basically the shortest proof of the inequality $\fm(\ell, \Sigma)< \infty$ we have. It might be interesting  to find a shorter argument.

\vspace{2mm}
Most  Tur\'an-type problems have their saturation counterparts, where the goal is to \textit{minimize} the number of edges in an inclusion-maximal $\cF$-free graph, see the survey~\cite{faudree2011survey} by Faudree, Faudree, and Schmitt. For the special case when $\cF$ is a family of cycles, see~\cite{demidovich2023cycle, furedi2013cycle, korandi2017saturation}. The study of planar saturation numbers has been recently initiated by Clifton and Salia~\cite{clifton2024saturated}, see also~\cite{barat2025number}. Note that if $\cF=\CC$ with $\length>3$, then every \m{} contains at least $n-1$ edges, which is tight as witnessed by stars. However, if we consider only $2$-connected plane graphs, i.e. such graphs that all their faces are bounded by cycles, then the problems becomes less trivial.

\begin{Question}
	What is the minimum number ${\rm sat}_{\mathcal{P}}^{2{\rm\mbox{-}con}}(n,\CC)$ of edges  in a 2-connected \m{} graph on $n$ vertices?
\end{Question}

\noindent
A direct application of Euler's formula yields that ${\rm sat}_{\mathcal{P}}^{2{\rm\mbox{-}con}}(n,\CC) \ge (1-2/\fm(\length, \plane))^{-1}(n-2)$ for all $n\ge \length \ge 3$. It would be interesting to improve this lower bound asymptotically.

\vspace{2mm}
Note that $\fm(\length, \Sigma)$ is defined as the maximum length of a facial cycle of an inclusion-maximal graph with girth \textit{at least $\length$} embedded on $\Sigma$, while one could ask for a variant of this problem for graphs with girth \textit{exactly $\length$}. We claim that these two problems have the same answer. Indeed, consider a graph $G$ with girth at least $\length$ embedded on $\Sigma$ with a face $F$ bounded by a cycle of length $\fm(\length, \Sigma)$. Draw a cycle of length $\length$ in any face of $G$ but $F$ and add edges arbitrarily until the resulting graph $G'$ is an inclusion-maximal $\CC$-free graph. It is easy to see that $\fm(G')\ge\fm(G)$ since $F$ is still a face of $G'$, while the girth of $G'$ equals $\length$, as desired.

\vspace{2mm}
Finally, let us note that the relations between lengths of facial cycles in plane graphs and their other parameters including radius or diameter were also considered, see Ali, Dankelmann, and Mukwembi~\cite{ADM} and Du Preez~\cite{P}, respectively. See also a paper by Fern\'andez, Sieger, and Tait \cite{FST} on planar subgraphs of given girth in planar graphs.

\vspace{3mm}

\noindent
{\bf \large Acknowledgements.}
Research was supported in part by the DFG grant FKZ AX 93/2-1.

\appendix

\section{Upper bound on an arbitrary closed surface: a shorter proof} \label{SA}

Let $g\ge 1$ and $\ell\geq 3$ be integers and $\Sigma$ be a surface, $\Sigma \in \{\mathbb{S}_g, \N_g\}$. In this section, we provide an alternative argument to show that $\fm(\length, \Sigma)$ is finite. The resulting upper bound is worse than the one given by \Cref{thm_orient_surf}, but the proof is much shorter.

\begin{Theorem}\label{thm_gen_surf}
	Let $g\ge 1$ and $\ell\geq 3$  be  integers  and $\Sigma$ be a surface,  $\Sigma \in \{\mathbb{S}_g, \N_g\}$. Then
	\begin{equation*}
		\fm(\length, \Sigma) \le \big((4g+4)^2\ell\big)^{\ell}.
	\end{equation*}
\end{Theorem}

Our proof utilizes the celebrated Ringel's theorem, see e.g. \cite[Theorem~4.4.7]{MT01}, which gives a necessary and sufficient condition for the existence of an embedding of a complete bipartite graph on a closed surface.

\begin{Theorem}[Ringel \cite{Rin1965a,Rin1965b}] \label{Rin}
	Let  $s,t \ge 3$ and $g\ge 0$ be intergers. Then $K_{s,t}$ can be embedded on $\mathbb{S}_g$ if and only if $g \ge \ceil{\frac{(s-2)(t-2)}{4}}$. Moreover, $K_{s,t}$ can be embedded on $\N_g$ if and only if $g \ge \ceil{\frac{(s-2)(t-2)}{2}}$.	
\end{Theorem}

Recall that the class of graphs that can be embedded on $\Sigma$ is closed under taking minors, see e.g. \cite[Section~5.9]{MT01}. Hence, \Cref{Rin} also implies that a graph which has a `large' bipartite clique as a minor cannot be embedded on $\Sigma$. Here we only need the following simple corollary of this general result.

\begin{Proposition} \label{Rin3}
	Let $g \ge 0,\, t=4g+3$ be integers, $\Sigma \in \{\mathbb{S}_g, \N_g\}$, and $H$ be a graph that has $K_{3,t}$ as a minor. Then $H$ cannot be embedded on $\Sigma$.
\end{Proposition}

We shall also need the following definitions. A {\it spider} is a tree that is a subdivision of a star, i.e., a tree with exactly one vertex of degree at least three, called the {\it head} of the spider.  A {\it leg} of a spider is a path with endpoints that are the head and a leaf of the spider.
A $t$-{\it subspider} of a spider  $S$ is a spider obtained by taking the union of some $t$ legs of $S$.
A {\it pseudo-spider} with \textit{head} $u$ and \textit{leaf-set} $U$ is a union of some not necessarily edge-disjoint $u,u'$-paths, $u'\in U$, that are called its \textit{legs}. Note that as a graph, a pseudo-spider may not be a tree, since the vertices of its leaf-set may be of degree larger than 1.

\begin{Lemma} \label{spiders_cross}
	Let $g \ge 0,\, t=4g+3$ be integers, $\Sigma \in \{\mathbb{S}_g, \N_g\}$, $G$ be a graph embedded on $\Sigma$, and $C$ be its facial cycle. Let $S_1$ be a spider and $S_2$ be a pseudo-spider in $G$ with $t$ leaves on $C$ each such that their leaves alternate on $C$. Then $V(S_1)\cap V(S_2)\neq \varnothing$.
\end{Lemma}

\begin{proof}
	Assume that $V(S_1)\cap V(S_2) = \varnothing$.
	Let $u_1, v_1, \ldots, u_t, v_t$ be the leaves of $S_1$ and $S_2$ that appear alternatingly in this order in $C$.  
	Consider a star  $S$ with center $w$ and  $t$ leaves $w_1, \ldots, w_t$, where all the vertices are not the vertices of $G$.
	Let $T$ be a tree obtained from $S$  by adding edges $w_ju_j$ and $w_jv_j$, $j=1, \ldots, t$.  We can embed $T$ on the face bounded by $C$.
	Then we see that the graph $S_1\cup S_2\cup T$ embedded on $\Sigma$ has $K_{3,t}$ as a minor. This contradicts \Cref{Rin3}.
\end{proof}

\vspace{0mm}

\begin{proof}[Proof of Theorem \ref{thm_gen_surf}]
	Let $g \ge 1$, $t=4g+3$ and $\ell\ge 3$ be integers, $\Sigma \in \{\mathbb{S}_g, \N_g\}$, $G$ a \ms, and $C$ it its facial cycle.  Assume  for the contrary that $\|C\| >  \big((t+1)^2\ell\big)^{\ell}$.    We shall find a spider $S_1$ with head $z$ and a pseudo-spider $S_2$ with $t$ leaves on $C$ and legs of length less than $\ell$ such that there are $t \ell$ leaves  of $S_1$ on $C$ between any two consecutive leaves of $S_2$ and $z\not \in V(S_2)$. We shall argue that there is a leg of $S_2$ with at least $\ell$  distinct vertices from $S_1$, i.e., more than the total number of vertices in that leg. This will result in a final contradiction.

	{\it Building $S_1$.}   Consider a subset $W\subset V(C)$  such that the distance between its vertices on $C$  is at least  $\ell+1$ and such that $|W| \ge \big((t+1)^2\ell\big)^{\ell-1}$. Fix a vertex $w\in W$ and a $(W,w)$-tree $T$, which exists by \Cref{Wwtree}.
	On the one hand, note that $W \subseteq V(T)$ by definition. On the other hand, recall that the pairwise distance between the vertices of $V(C)$ is at most $\ell-2$ by  \Cref{cl0}. Thus the height of $T$ as a rooted tree with the root $w$ is at most $\ell-1$. Hence, if the maximum degree of $T$ is at most $d \coloneqq (t+1)^2\ell$, then $|V(T)|< d^{\ell-1}\le |W|$, 
	a contradiction. Therefore, there exists a vertex $z \in V(T)$ of degree more than $(t+1)^2\ell$ in $T$.
	By following the edges incident to $z$ to the respective leaves of $T$, we see that $T$ contains a spider $S_1$ with head $z$ and with $(t+1)^2\ell$ legs such that each of its leaves is also a leaf of $T$. Denote the set of leaves of $S_1$ by $W'$ and observe that $W' \subseteq W$ by construction. 
	
	{\it Building $S_2$.} By \Cref{center_avoiding}, there is a set of vertices $U$ obtained by picking one vertex between every two consecutive on $C$ vertices of $W'$ such that their distance to $z$ is greater than $\ell/2-1$. Let $U' \subset U$ consists of every $t\ell^{\rm th}$ vertex of $U$ in their order on $C$. Recall that $|U|=|W'|= (t+1)^2\ell$, and thus $|U'|\ge t+1$. Pick a vertex $u\in U'$ and a pseudo-spider $S_2$ with head $u$ and $t$ legs that are shortest paths in $G$ from $u$ to some $t$ vertices in $U'\sm\{u\}$.
	Observe that  $z\not\in V(S_2)$ because otherwise $z$ is on some leg of $S_2$, say with a leaf $u'$, and thus $\dist_G(u,u')=\dist_G(u,z)+\dist_G(u',z)> 2(\frac{\ell}{2}-1)$ which contradicts \Cref{cl0}.

	{\it Crossings between $S_1$ and $S_2$.}
	Recall  that there are at least $t\ell$ leaves of $S_1$ between any two consecutive leaves of $S_2$ by construction. Therefore,  there are  $t\ell$  pairwise edge-disjoint $t$-subspiders of $S_1$ that are leaf-alternating with $S_2$.  Each of these subspiders share a common vertex with $S_2$ by \Cref{spiders_cross}. Moreover, these vertices are different for different $t$-subspiders because $z \notin V(S_2)$. By the pigeonhole principle, there is a leg of $S_2$ that contains at least $(t\ell)/t = \ell$ different vertices. However, the legs of $S_2$ have lengths less than $\ell$ by \Cref{cl0}. This contradiction yields that $\|C\| \le  \big((t+1)^2\ell\big)^{\ell}$, as desired.
\end{proof}

\end{document}